\documentclass[12pt, reqno]{amsart}

\usepackage{amsmath, amsthm, amscd, amsfonts, amssymb}

\usepackage{tikz}
\usetikzlibrary{patterns,matrix,arrows}
\usetikzlibrary{arrows}

\usepackage{amssymb}
\usepackage{amsfonts}
\usepackage{amsmath}

\usepackage{graphicx}

\usepackage{enumerate}

\usepackage{tikz}
\usetikzlibrary{arrows}

\newtheorem{theorem}{Theorem}[section]

\newtheorem{proposition}[theorem]{Proposition}
\newtheorem{corollary}[theorem]{Corollary}
\theoremstyle{definition}
\newtheorem{definition}[theorem]{Definition}
\newtheorem{example}[theorem]{Example}

\theoremstyle{remark}
\newtheorem{remark}[theorem]{Remark}
\numberwithin{equation}{section}

\newcommand{\g}{g}

\newcommand{\Cstar}{C^\ast}
\newcommand{\C}{$C^\ast$}
\newcommand{\G}{\mathcal{G}}
\newcommand {\N}{\mathbb{N}}

\usepackage{datetime}

\begin{document}

\title[Interval maps, branching systems and relative graph algebras]{Unifying interval maps and branching systems with applications to relative graph C*-algebras}


\author{C.\ Correia Ramos}
\address{Centro de Investiga\c c\~ao em Matem\'atica e
Aplica\c c\~oes,
 Department of Mathematics,
Universidade de \'{E}vora,
R.\ Rom\~{a}o Ramalho,
59, 7000-671 \'{E}vora,\\
 Portugal}
   \email{ccr@uevora.pt}
   
 \author{D.\ Gon\c{c}alves}
\address{Departamento de Matem\'atica,
Universidade Federal de Santa Catarina,
Florian\'opolis, 88040-900\\
 Brazil}
   \email{daemig@gmail.com}

\author{N.\ Martins}
\address{Department of Mathematics, CAMGSD,
Instituto Superior T\'{e}cnico, University of Lisbon,
Av.\ Rovisco Pais 1, 1049-001 Lisboa, Portugal}
\email{nmartins@math.tecnico.ulisboa.pt}

\author{P.R.\ Pinto}
\address{Department of Mathematics, CAMGSD,
Instituto Superior T\'{e}cnico, University of Lisbon,
Av.\ Rovisco Pais 1, 1049-001 Lisboa, Portugal}
\email{ppinto@math.tecnico.ulisboa.pt}

\begin{abstract} We describe Markov interval maps via branching systems and develop the theory of relative branching systems, characterizing when the associated representations of relative graph C*-algebras are faithful. When the Markov interval maps $f$ have escape sets, we use our results to characterize injectivity of the associated relative graph algebra representations, improving on previous work by the first, third, and fourth authors. 
\end{abstract}

\maketitle

{\it Mathematics Subject Classification}: Primary 46L05, 37B10; 
Secondary 37E05.

{\it Keywords}: Relative graph C$^\ast$-algebra; Branching systems; Representation; Transition matrix; Interval map.


\section{Introduction}

Our main goal is to find faithful representations of relative graph algebras, on separable Hilbert spaces, that naturally appear when we consider the open dynamics systems arising from Markov interval maps with escape points. Our study indicates a unifying perspective among the subjects of interval maps, (relative) branching systems, and relative graph algebras, whenever the underlying graphs are finite. This is built starting from previous work on the construction of branching systems and the associated representations of graph algebras by the second author \cite{CantoGoncalves, FGG, ghr, GR2, Daniel}, and by the three other authors in relating open dynamical systems and representations of relative graph algebras (underlying the transition matrices) on the orbit spaces \cite{CMP, RMP10, RMP20, RMP11}.
  
For a directed (finite) graph $\G$ -- and $V$ a subset of vertices in $\G$ -- the relative graph \C-algebra $\Cstar(G,V)$, introduced by Muhly and Tomforde in \cite{Tom}, is a universal \C-algebra that allows the simultaneous study of the associated graph algebra (Cuntz-Krieger algebra if $V$ is the full set of vertices $\G^0$) and Toeplitz algebra (when $V=\emptyset$). Among the applications of relative graph \C-algebras we mention that relative graph \C-algebras are the key ingredient in the study of KMS states associated to graph algebras done in \cite{toke} and have been generalized to finitely aligned higher rank graphs, see \cite{Sims}. 

Relative graph \C-algebras have purely algebraic siblings, called Cohn-Leavitt path algebras, which are the base of the theory of Leavitt path algebras, see \cite{AASbook}. It is from the purely algebraic context that we draw inspiration to define an analytical {\it relative branching system} (the algebraic version is defined in  \cite{CantoGoncalves}). Then, every such relative branching system (which arises from a tuple $(\G,V)$, of a graph and a subset of vertices, on a measure space $(X,\mu)$) yields a representation
\begin{equation}\label{eq1}
  \pi:\ \Cstar(\G,V)\to B(L^2(X,\mu))
\end{equation}
of the relative \C-algebra $\Cstar(\G,V)$ on the Hilbert space $L^2(X,\mu)$. Building on the work in \cite{GLR} for graph \C-algebras, our study leads to a characterization of those representations that are faithful (in terms of the underlying branching systems). These results have interest in their own right, as the current streamline of research on the faithfulness of representations of ultragraph algebras and higher rank graph algebras from branching systems shows, see \cite{FGG, FGJKP, JorgensenFarsi, GLR2, ghr}.

As we mentioned above, we will connect (and apply) the theory of branching systems with the open dynamics of interval maps. More precisely, we  consider the class of open dynamics arising from interval maps $g: \hbox{dom}(\g)\to I$ (with dom$(g)\subset I$) the domain of $\g$ with escape points, so that the orbit of a point
$x$ can be in the \textit{escape set} of $g$; namely, $g^{k}(x)\in I$
does not belong to the domain of $f$ for a certain $k\in \mathbb{N}$, see \cite{RMP11}. Then the orbit of $x$ generates a Hilbert space 
$$H_x=\ell^2(\hbox{orb}(x)).$$
The study of the open dynamics $(I,g)$ uses its transition matrix $A_g$ as well as the escape transition matrix $\widehat{A}_g$, that encompass not only the transitions among the Markov subintervals but also the transitions into the escape set $E_g$, see \cite{RMP11}. The underlying symbolic dynamics allows us to build a set $V_x$ of vertices in the graph $\G_{A_g}$ of the transition matrix $A_g$ (see \eqref{eqvx} below) and a representation
\begin{equation}\label{eqint2}
\nu_x:\ \Cstar(\G_{A_g},V_x)\to B(H_x)
\end{equation}
of the relative graph algebra $\Cstar(\G_{A_g},V_x)$ on $H_x$, see \cite{RMP20}. 
In this paper, we construct a relative branching system associated with $(\G_{A_g},V_x)$ on $(X,\mu)$, where $X$ is the orbit of $x$ (see \eqref{eqXmu}) and $\mu$ is the counting measure. Then, we show that the representation of $\Cstar(\G_{A_g},V_x)$ arising from this branching system (as in \eqref{eq1}) and $\nu_x$ (as in \eqref{eqint2}) coincide.

Once the representation $\nu_x$ (see \eqref{eqint2}) is translated into the relative branching systems framework we can apply, for $\nu_x$, our characterization of faithful representations arising from relative branching systems. As a corollary, we obtain that $\nu_x$ is a faithful representation, for example, when $A_g$ is an irreducible matrix, even if $g$ is not an expansive map (thus improving \cite[Thm.\ 3.1]{RMP20}). In particular, we can produce faithful representations of Toeplitz algebras (the cases with $V_x=\emptyset$), which is a non-simple \C-algebra.

The plan for the rest of the paper is as follows. 
In Sect.\ \ref{secbackg} we review some necessary background, starting in Subsect.\ \ref{secmarkovescape} with the elements of the framework of the open dynamics provided by the interval maps with escape sets, followed by the notion of relative graph algebras and a revision of the extended graph (Subsect.\ \ref{relativo}). In Subsect.\ \ref{subrepsescape}, we write the representations of the relative graphs algebras on the Hilbert spaces $H_x$ attached to the orbits obtained in \cite{RMP20}, and state \cite[Thm.\ 3.1]{RMP20} as we will generalize it for other interval maps (not necessarily aperiodic or non-expansive). 

In Sect.\ \ref{BSU}, we put forward a notion of relative branching system, on a measure space $(X,\mu)$, associated to a pair $(\G,V)$. Our definition enables us to prove, in Prop.\ \ref{reprelativebs}, that every such relative branching system gives rise to a representation $\pi$ of the relative graph \C-algebra $\Cstar(\G,V)$. 
Then, we use the extended graph $E(\G)$ associated with the tuple $(\G,V)$ to naturally relate the above representation $\pi$ of $\Cstar(\G,V)$ and the representation of the graph algebra $\Cstar(E(\G))$ arising from a certain branching system associated with $E(\G)$, see Prop.\ \ref{repcommute}. This enables us to obtain one of the main results of this paper, Thm.\ \ref{jacare11}, which characterizes when the representation $\pi$ is faithful (which is an analytical version of \cite[Thm.\ 5.6]{CantoGoncalves} and a generalization of \cite[Thm.\ 3.1]{GLR} to relative graph algebras). In Sect.\ \ref{sectintmapsBS}, we apply Thm.\ \ref{jacare11} to the framework of interval maps $(I,g)$, where every point $x\in I$ leads to a construction of a relative branching system as shown in Prop.\ \ref{propintmapsbsysts}. Then, in Thm.\ \ref{repequivalence} we show that the representation $\pi$ arising from the relative branching system and the one $\nu_x$ reviewed in Thm.\ \ref{rmp20theorem} coincide. Thanks to Thm.\ \ref{jacare11}, we are able to conclude when $\nu_x$ is faithful (see Thm.\ \ref{injectiveMap}) even in contexts without an aperiodicity or expansivity assumption on the interval map $g$. This is applied in Sect\ \ref{secapps} to elaborate several faithful representations of relative graph algebras that are not covered by Thm.\ \ref{rmp20theorem}.


\section{Background}

In this paper we denote the natural numbers by $\N$ and the set $\{0,1,2, \ldots\}$ by $\N_0$.

\label{secbackg}

\subsection{Markov interval maps with escape sets}
\label{secmarkovescape}

Given $n\in \mathbb{N}$, let
\[\Gamma=\{c_{0},c_{1}^{-},c_{1}^{+},...,c_{n-1}^{-},c_{n-1}^{+},c_{n}\}\]
be an ordered set of (at most) $2n$ real numbers such that
\begin{equation}
c_{0}<c_{1}^{-}\leq c_{1}^{+}<c_{2}^{-}\leq ...<c_{n-1}^{-}\leq
c_{n-1}^{+}<c_{n}\text{.}
\label{eq2}
\end{equation}

Given $\Gamma $ as above, we define the collection of closed intervals
$C_{\Gamma }=\{I_{1},...,I_{n}\}$, with
\begin{equation}
I_{1}=\left[ c_{0},c_{1}^{-}\right],\ ...,\ I_{j}=\left[
c_{j-1}^{+},c_{j}^{-}\right],..., I_{n}=\left[c_{n-1}^{+}, c_{n}\right].
\label{eq3}
\end{equation}%
We also consider the collection of open intervals $\{E_{1},...,E_{n-1}\}$,
with
\begin{equation}
E_{1}=\left] c_{1}^{-},c_{1}^{+}\right[,\ ...,\ E_{n-1}=\left]
c_{n-1}^{-},c_{n-1}^{+}\right[,  \label{eq4}
\end{equation}%
in such a way that $I:=[c_{0},c_{n}]=\left(\cup _{j=1}^{n}I_{j}\right)
\bigcup \left(\cup _{j=1}^{n-1}E_{j}\right)$.


We now specify the interval maps for which we can construct partitions of
the interval $I$ as in \eqref{eq2}, \eqref{eq3} and \eqref{eq4}.

\begin{definition}[cf.\ \cite{CMP}]
Let $I\subset \mathbb{R}$ be an interval. A measurable map $\g$ is called a Markov interval map (equivalently it is said to be in the class $M(I)$) if it satisfies the following properties:

\begin{enumerate}
\item[(P1)] \textrm{[}Existence of a finite partition in the domain of $\g$\textrm{]} There is a partition $C=\left\{I_{1},...,I_{n}\right\}$ of closed
intervals with $\#\left(I_{i}\cap I_{j}\right)\leq 1$ for $i\neq j$,
$\mathrm{dom}(\g)=\bigcup _{j=1}^{n}I_{j}\subset I$, $\min(I_1)=\min(I)$, $\max(I_n)=\max(I)$, $\mathrm{im}(\g)=I$ and $\g\vert_{I_i}$ is monotonous (for $i=1,...,n$).

\item[(P2)] \textrm{[}Markov property\textrm{]} For every $i=1,...,n$ the
set $\g(I_{i})\cap \left(\bigcup_{j=1}^{n}I_{j}\right)$ is a non-empty union
of intervals from $C$.

\end{enumerate}
\label{def1}
Furthermore, $\g\in M(I)$ is said to be
\begin{itemize}
\item[(P3)] \emph{Expansive} if $\g_{|I_{j}}\in \mathcal{C}^{1}(I_{j})$, is monotone and $|\g_{|I_{j}}^{\prime }(x)|>b>1$, for every $x\in
I_{j},j=1, ..., n$, and some $b$.

\item[(P4)] \emph{Irreducible} if for every $i,j$ there is a natural number $q$ such that\\ $I_i\, \subset\,
\g^{q}(I_{j})$.
\item[(P5)] \emph{Aperiodic} if for every interval $I_{j}$ with
$j=1,...,n$ there is a natural number $q$ such that $\mathrm{dom}(\g)\subset
\g^{q}(I_{j})$.
\end{itemize}

\end{definition}

\begin{remark}
In \cite{CMP, RMP20} a map $\g$ is said to be a Markov map if it satisfies conditions $(P1), (P2), (P3)$ and $(P5)$ above. Since we will generalize the results in \cite{RMP20} to maps that need to satisfy only $(P1)$ and $(P2)$, we decided to call such maps of Markov maps.
\end{remark}

The minimal partition $C$ satisfying Def.\ \ref{def1} is denoted by $C_{\g}$. We remark that the transitions between the intervals can be encoded in the so-called (Markov) transition $n\times n$ matrix
$A_{\g}=(a_{ij})$, given by:
\begin{equation}
a_{ij}=\left\{
\begin{array}{l}
1\text{ if }\g(\mathring{I}_{i})\supset \mathring{I}_{j}, \\
0\text{ otherwise}
\end{array}
\right.  \label{mx}
\end{equation}
where $\mathring{A}$ denotes the interior of the set $A$.
A map $\g\in M(I)$ (together with the minimal
partition $C_{\g}=\left\{ I_{1},...,I_{n}\right\}$) determines:

\begin{enumerate}[(i)]
\item The $\g$-invariant set $\Omega _{\g}:=\{x\in I:\ \g^{k}(x)\in \mathrm{dom}%
(\g)$ for all $k=0,1,...\}$.

\item The collection of open intervals $\left\{ E_{1},...,E_{n-1}\right\} $,
such that ${I\setminus \bigcup _{j=1}^{n}I_{j}}=\bigcup _{j=1}^{n-1}E_{j}$.

\item The transition matrix $A_{\g}=\left( a_{ij}\right)_{i,j=1,..,n}$.
\end{enumerate}

Let $I$ be an interval and $\g\in M(I)$. Once we chose a
Markov partition $C$, then there exists a unique set $\Gamma$ of boundary points
satisfying \eqref{eq2}, \eqref{eq3} and \eqref{eq4}. From the Markov property we have that $\g\left(\Gamma \right) \subset \Gamma$. We define the generalized orbit of a point $x\in I$ as 
\[R_\g(x):=\{y\in I: \g^n(x)=\g^m(y) \text{ for some } n,m \in \mathbb{N}_0 \}. \]


The set
\begin{equation}
\tilde{E}_{\g}:=I\setminus \Omega_g = \bigcup_{k=0}^{\infty }\g^{-k}\left(\bigcup _{j=1}^{n-1}E_{j}\right)
\end{equation}
is called the \textit{escape set}, but for our purposes it is enough to consider the subset of $\tilde{E}_g$ given by 
\begin{equation}
{E}_{\g}:=\bigcup_{k=1}^{\infty }\g^{-k}\left(\bigcup _{j=1}^{n-1}E_{j}\right).
\end{equation}
Every point in $E_\g$ will eventually
fall, under iteration of $\g$, into some interval $E_{j}$
(where $\g$ is not defined) and the iteration process ends.
Observe that $x$ is in $E_{\g}$ of $\g$ if, and only if, there is $k\in
\mathbb{N}$ such that $\g^{k}\left(x\right) \notin \mathrm{dom}(\g)$.

Let $\g\in M(I)$ be such that $E_{\g}\neq \emptyset$. This means
that there is at least one non empty open interval
$E_{j}=\left]c_{j}^{-},c_{j}^{+}\right[ $,
with $c_{j}^{-}\neq c_{j}^{+}$, with $j\in\{1,...,n-1\}$.
The non-empty open subinterval $E_{j}$ is called an \textit{escape interval}.

In order to describe symbolically the escape orbits, we extend the symbol
space adding a symbol for each escape interval $E_{j}$, which will represent
an end for the symbolic sequence. For each escape interval $E_{j}$ we
associate a symbol $\widehat{j}$ to distinguish from the symbol associated to
the interval $I_{j}$ of the partition. That is, we consider the symbols ordered by:
\begin{equation}
1<\widehat{1}<2<\widehat{2}< ... < n-1< \widehat{n-1} <n\text{.}
\label{eqorder}
\end{equation}

If $E_{j}$ is not an interval, that is $E_{j}=\emptyset $, then there is no symbol
$\widehat{j}$. Moreover, we define
\begin{equation}
\Sigma_{E_\g}=\left\{\widehat{j}:  E_{j}\neq \emptyset,\ j\in\left\{1,..., n-1\right\}
\right\}.
\label{eqhatsymbols}
\end{equation}

For
every $y\in E_{\g}$ there is a least natural number $\tau \left( y\right) $
such that $\g^{\tau \left(y\right)}\left(y\right)\notin\mathrm{dom}\left(
\g\right)$, which means that, $\g^{\tau \left( y\right)}\left(y\right)\in
E_{j}$, for some $j$ such that $E_{j}\neq\emptyset$.
The final escape point, for the orbit of $y$, is then denoted by
$e\left( y\right) :=\g^{\tau \left( y\right)}\left( y\right) $ and the final
escape interval index is denoted by $\iota \left(y\right) $, that is, if
$\g^{\tau \left( y\right) }\left( y\right) \in E_{j}$ then $\iota \left(
y\right) =\widehat{j}$.

Thus we have an index set $\{1,...,n\}\bigcup\Sigma_{E_\g}$ which is ordered as in \eqref{eqorder} and \eqref{eqhatsymbols}.
To deal with the possible transitions from Markov transition intervals to
escape intervals we define the escape transition matrix $\widehat{A}_{\g}$ as follows.

\begin{definition}[see \cite{RMP10, RMP11}]
Given the transition matrix $A_\g$ as in \eqref{mx}, we define a matrix
 $\widehat{A}_{\g}=\left( \widehat{a}_{ij}\right)$ indexed by $\{1,...,n\}\cup\Sigma_{E_\g}$ such that
\begin{equation}
\widehat{a}_{ij}=\left\{
\begin{array}{ll}
a_{ij}&\text{ if } i,j\in\{1,...,n\}, \\
1&\text{ if } i\in \{1,...,n\},\ j=\hat{k}\in \Sigma _{E_\g}\ \hbox{and}\ \mathring{I}_{i}\cap \g^{-1}\left(
E_{k}\right) \neq \emptyset,\\
0&\text{  }  \hbox{otherwise}.
\end{array}%
\right.
\label{mmx}
\end{equation}
\label{defhatA}
\end{definition}
For row and column labeling, the matrix $\widehat{A}_\g$ is defined by considering the order in \eqref{eqorder}.

\subsection{Relative graph algebras.}\label{relativo}

We recall that a (directed) graph $\G$ is a quadruple $(\G^0,\G^1, r,s)$ consisting of two countable sets $\G^0$ and $\G^1$, and two maps $r,s:\G^1\to\G^0$. We think of $\G^0$ as the set of vertices of $\G$, and every $e\in\G^1$ is regarded as an arrow pointing from $s(e)$ to $r(e)$. 
For $v\in\G^0$, we call $v$ a sink if $s^{-1}(v)=\emptyset$, and call $v$ a source if $r^{-1}(v)=\emptyset$. 
In this paper we only work with finite graphs.  

We also recall that a path of length $n$ in $\G$ is a tuple $(e_i)_{i=1}^n\in \Pi_{i=1}^n \G^1$ such that $r(e_i)=s(e_{i+1})$ for $i=1,...,n-1$. For a path $\gamma=\gamma_1\ldots \gamma_n$ we define $\gamma^0:=\{v\in \mathcal{G}^0:v=s(\gamma_i), i=1,\ldots,n\}$. The path $(e_i)_{i=1}^n$ is called a cycle if $s(e_1)=r(e_n)$. A cycle $(e_i)_{i=1}^n$ is called simple if $r(e_i)\not=r(e_j)$ for all $i\not= j$. Finally we say that a cycle $(e_i)_{i=1}^n$ has no exit if $s^{-1}(s(e_i))=e_i$ for all $i=1,...,n$.

For any $V\subset \G^0$, the {\it relative graph algebra} $\Cstar(\G,V)$ (see \cite[Def.\ 3.4]{Tom}) is the universal \C-algebra generated by partial isometries $\{s_e: e\in \G^1\}$ and mutually orthogonal projections $\{p_v: v\in \G^0\}$ satisfying the following set of relations $\mathcal{R}_V$:

\begin{eqnarray}
  s_e^\ast s_e &=&  p_{r(e)}, \label{eqsesepr}\\
   s_e s_e^\ast &\leq& p_{s(e)},\ \hbox{for all}\ e\in \G^1, \ \hbox{and} \label{eqseseps}\\
  p_v &=& \sum_{e\in \G^1:\ s(e)=v} s_e s_e^\ast,  \label{eqpv=}
\end{eqnarray}
 for all $v\in V$ such that $\{e\in \G^1: s(e)=v\}\not=\emptyset$.

As shown in \cite{Tom}, every relative graph C*-algebra can be seen as a graph C*-algebra, by using the associated extended graph. We recall these results below.

\begin{definition}\label{extendedgraph} Let $\mathcal{G}$ be an graph and $V\subseteq \mathcal G^0$. Define the graph $E(\G)=(E^0,E_1,r_{E(\G)},s_{E(\G)})$ so that $E^0= \G^0 \cup \{v':v\in E^0, v \notin V\}$ and $E^1= \G^1 \cup \{e': e\in \G^1, r(e)\notin V\}$. Furthermore, for $e\in \G^1$, let $r_{E(\G)}(e)=r(e)$, $s_{E(\G)}(e)=s(e)$ and, if $r(e)\notin V$, let $r_{E(\G)}(e')=r(e)'$ and $s_{E(\G)}(e')=s(e)$.
$E(\G)$ is called the {\it extended graph} of $(\G,V)$.
\end{definition}

We recall the isomorphism between $\Cstar(\G,V)$ and $ \Cstar(E(\G))$ below.

\begin{proposition}\label{isorelat}(cf.\ \cite{AASbook} Theorem 1.5.18)

 Let $\mathcal{G}$ be a graph and $V\subseteq \mathcal G^0$ and let $E$ be the extended graph $E(\G)$ as defined above. Then, there is an isomorphism $\phi:C^*(\G,V) \rightarrow C^*(E)$ such that 
 $$\phi(p_v)=\begin{cases} p_v+p_{v'} \text{ if } v\notin V\cr 
 p_v\ \text{otherwise} \end{cases}\ \hbox{and}\ \ \phi(s_e)=\begin{cases} s_e \text{ if }\ r_E(e)\in V \cr s_e+s_{e'}\ \text{if}\ r(e)\notin V. \end{cases}$$
 
 Moreover, the inverse of $\phi$ is given by an isomorphism $\psi: C^*(E)\rightarrow C^*(\G,V)$ such that 
 
 $$\psi(p_v)=\begin{cases}p_v\ \text{if}\ v\in V,\cr 
  \sum_{e\in s^{-1}(v)} s_e s_e^*\ \text{if}\ v\notin V,\end{cases},\ \ \psi(p_{v'})=p_v-\sum_{e\in s^{-1}(v)}s_es_e^*$$  
  and   $$\psi(s_e)=\begin{cases}
  s_e\ \text{if}\ r(e)\in V, \cr
  \sum_{f\in s^{-1}(v)}s_fs_f^*\ \text{if} \ r(e)=v \notin V,\ \ 
  \text{and}\ \psi(s_{e'})=s_e (p_v-\sum_{f\in s^{-1}(v)}s_fs_f^*).
  \end{cases}$$ 
\end{proposition}

In our work, we are interested in graph algebras in connection with Cuntz-Krieger algebras. More precisely, for any
$n\times n$ matrix $A=(a_{ij})$ with entries in $\{0,1\}$, we construct a
directed graph $\G_A=(\G_A^1,\G_A^0, r,s)$ with
\begin{equation}
\G_A^0=\{1,...,n\},\ \G_A^1=\{e_{ij}: i,j\in\G_A^0, a_{ij}= 1\}\ \hbox{with}\ s(e_{ij})=i,\ r(e_{ij})=j.
\label{eqGA}
\end{equation}
The Cuntz-Krieger algebra $\mathcal{O}_A$ is then isomorphic to the graph \C-algebra $\Cstar(\G_A)$, see \cite[Prop.\ 4.1]{Rae2}.

\subsection{Representations of relative graph algebras arising from interval maps}
\label{subrepsescape}

We recall next how to obtain representations of a relative graph algebra associated with $\G_{A_\g}$, where $g\in M(I) $ is such that $E_{\g}\neq \emptyset$.

Let $H_{x}=\ell^2(R_\g(x))$, with $x\in E_{\g}$, be the Hilbert space with canonical base
\begin{equation*}
\left\{ \left\vert z\right\rangle: \g^{k}\left(z\right)= e\left(x\right)\ \hbox{for some}\ k\in \mathbb{N}_{0}\right\}.
\end{equation*}
 Note that there is a
special vector basis which is $\left\vert e\left(x\right)\right\rangle $.
The rank one projection on the 1-dimensional space $\mathbb{C}\left\vert z\right\rangle$
is denoted by $P_{z}$, or as usual in Dirac notation, $P_{z}=\left\vert
z\right\rangle \left\langle z\right\vert $.
We will write $S_{ij}$ instead of $S_{e_{ij}}$. For vectors $|y\rangle$ such that $y\in R_\g(x)$, we define
\begin{eqnarray}
  S_{ij} |y\rangle &=& \chi_{I_j}(y)\ | \g_i^{-1}(y)\rangle,\quad (\hbox{if}\ a_{ij}=1), \label{eqSij} \\
  P_i|y\rangle&=&  \chi_{I_i}(y)\ |y\rangle. \label{eqPi}
\end{eqnarray}
It is clear that $S_{ij}$ are nonzero partial isometries for $e_{ij}\in\G_{A_\g}^1$ (i.e.\ $a_{ij}=1$) with
$$S_{ij}^\ast |y\rangle =\chi_{I_j}(\g(y))\chi_{I_i}(y)\ |\g (y)\rangle.$$


The following is the main result in \cite{RMP20}.

\begin{theorem}(\cite[Thm.\ 3.1]{RMP20})\label{rmp20theorem}
Let $A_\g$ be the transition matrix of an expansive, aperiodic, Markov interval map $\g$ and $\widehat{A}_\g$ its escape transition matrix. Suppose that $E_\g\not=\emptyset$.
 Consider its oriented graph $\G_{A_\g}$ and fix $V\subseteq \G_{A_\g}^0$. Let $x\in E_\g$.

 \begin{enumerate}
 \item If $\widehat{a}_{i \iota(x)}=0$ for $i\in V$, then $s_{e_{ij}}\mapsto S_{ij}$ and $p_i\mapsto P_i$ defined in \eqref{eqSij} and \eqref{eqPi}  yield a representation, $\nu_x$, of $\Cstar(\G_{A_\g}, V)$ on the Hilbert space $H_x$.
 \item If $\widehat{a}_{i \iota(x)}=0$ for $i\in V$ and $\widehat{a}_{k\, \iota(x)}=1$ for $k\notin V$, then the representation in part (1) is faithful.

\end{enumerate}
\label{thmnice100}
\end{theorem}

\section{Relative Branching Systems}\label{BSU}

In this section, we extend the algebraic notion of relative branching systems, which was introduced in \cite{CantoGoncalves}, to the analytical setting. In particular, for a relative graph, we define the associated relative branching system and describe faithfulness of the induced representation.

\begin{definition}\label{branchsystem}
Let $\mathcal{G}$ be an graph, $V\subseteq \mathcal G^0$, $(X,\mu)$ be a measure space and let $\{R_e,D_v\}_{e\in \mathcal{G}^1,v\in \mathcal{G}^0}$ be a family of measurable subsets of $X$. Suppose that
\begin{enumerate}
\item\label{R_e cap R_f =emptyset if e neq f} $R_e\cap R_f \stackrel{\mu-a.e.}{=}\emptyset$ if $e \neq f \in \mathcal{G}^1$;
\item $D_u \cap D_v \stackrel{\mu-a.e.}{=} \emptyset$ for all $u, v \in \mathcal{G}^0$;
\item $R_e\stackrel{\mu-a.e.}{\subseteq}D_{s(e)}$ for all $e\in \mathcal{G}^1$;
\item\label{D_v=cup_{e in s^{-1}(v)}R_e} $D_v\stackrel{\mu-a.e.}{=} \bigcup_{e \in s^{-1}(v)}R_e$ if $0 <\vert s^{-1}(v) \vert$ and $v\in V$; and
\item for each $e\in \mathcal{G}^1$, there exist two measurable maps $f_e:D_{r(e)}\rightarrow R_e$ and $f_e^{-1}:R_e \rightarrow D_{r(e)}$ such that $f_e\circ f_e^{-1}\stackrel{\mu-a.e.}{=} Id_{R_e}, f_e^{-1}\circ f_e\stackrel{\mu-a.e.}{=} Id_{D_{r(e)}}$, the measure $\mu \circ f_e$ in $D_{r(e)}$ is absolutely continuous with respect to $\mu$ in $D_{r(e)}$, and the measure $\mu \circ f_e^{-1}$ in $R_e$ is absolutely continuous with respect to $\mu$ in $R_e$.
\end{enumerate}
We call $\{R_e,D_v,f_e\}_{e \in \mathcal{G}^1,v \in \mathcal{G}^0}$ a \emph{relative $\mathcal{G}$-branching system} on $(X,\mu)$.
\end{definition}

\begin{remark}
If $V= \mathcal G^0$ above then we recover the usual definition of a branching system, as defined in \cite{GonRoy1}. Also, recall from \cite{GLR} that if $\alpha=\alpha_1\ldots  \alpha_n$ is a path in the graph $\mathcal G$ then $f_\alpha :=f_{\alpha_1}\circ \ldots \circ f_{\alpha_n}$ and  $f_\alpha^{-1} =f_{\alpha_n}^{-1}\circ \ldots \circ f_{\alpha_1}^{-1}$. Moreover, $f_\alpha ^k$ denotes the composition of $f_\alpha$ $k$ times. 
\end{remark}

 From now on, to simplify notation, we denote the Radon-Nikodym derivative $d(\mu \circ f_e)/d\mu$ by $\Phi_{f_e}$ and  the Radon-Nikodym derivative $d(\mu\circ f_e^{-1} )/d\mu$ by $\Phi_{f_e^{-1}}$.
Also, since the domain of $\Phi_{f_e^{-1}}$ and the domain of $ f_e^{-1}$ are $R_e$, we can also regard them as measurable maps on $X$ by simply extending then with value zero out of $R_e$, and so, for each $\phi\in \mathcal{L}^2(X,\mu)$, we can consider the function $\Phi_{f_e^{-1}}^{1/2}\cdot( \phi \circ f_e^{-1})$. Moreover, by extending $f_e$ and $\Phi_{f_e}$ by zero out of $D_{r(e)}$ we get the function $\Phi_{f_e}^{1/2}\cdot( \phi \circ f_e)$.

The procedure to obtain a representation of a relative graph C*-algebra from a given relative branching system follows the same outline as the usual graph algebra case (see \cite{GonRoy1} for example) and so we just describe it below (as a proposition).

\begin{proposition}
\label{repinducedbybranchingsystems}
Let $\mathcal{G}$ be a graph,  $V\subseteq \mathcal G^0$, and let $\{R_e,D_v,f_e\}_{e\in \mathcal{G}^1, v\in \mathcal{G}^0}$ be a relative $\mathcal{G}$-branching system on a measure space $(X,\mu)$. Then, there exists an unique representation $\pi:C^*(\mathcal{G},V) \to B(\mathcal{L}^2(X,\mu))$ such that 
$$\pi(s_e)(\phi)=\Phi_{f_e^{-1}}^{1/2}\cdot( \phi \circ f_e^{-1})\ \  \hbox{and}\  \ \pi(p_v)(\phi)=\chi_{D_v}\phi$$
for all $e \in \mathcal{G}^1, v \in \mathcal{G}^0$ and $\phi \in \mathcal{L}^2(X,\mu)$.
\end{proposition}

Our main goal in this section is to characterize when the representations above are faithful. For this, given a relative branching system of a graph $(\G,V)$, we first construct a branching system associated with the extended graph $E(\G)$, such that the representations arising from both branching systems coincide under the isomorphism (from Prop.\ \ref{isorelat}) between the relative graph algebra and the extended graph algebra. More precisely, we have the following.

\begin{definition}\label{extendedbs}
Let $\mathcal{G}$ be a graph,  $V\subseteq \mathcal G^0$, and let $\{R_e,D_v,f_e\}_{e\in \mathcal{G}^1, v\in \mathcal{G}^0}$ be a relative $\mathcal{G}$-branching system on a measure space $(X,\mu)$. Define the triple ${Q_e,B_v, g_e}$ associated with $E(\G)$ as follows.

For $v\in \G^0$, if $v\in V$ then let $B_v = D_v$. If $v\notin V$ then let $B_v = \bigcup_{e\in s^{-1}(v)} R_e$ and $B_{v'}= D_v \setminus \bigcup_{e\in s^{-1}(v)} R_e $.

For $e\in \G^1$, if $r(e)=v \in V$ then let $Q_e=R_e$. If $r(e)=v\notin V$ then define $Q_e=f_e(\bigcup_{f\in s^{-1}(v)} R_f)$ and $Q_{e'}=f_e(D_v \setminus \bigcup_{f\in s^{-1}(v)} R_f)$.  

Finally, for $e\in E(\G)^1$, define $g_e=f_e$ and $g_{e'}=f_e$.
\end{definition}

\begin{proposition}\label{reprelativebs} Let $\mathcal{G}$ be a graph,  $V\subseteq \mathcal G^0$, and let $\{R_e,D_v,f_e\}_{e\in \mathcal{G}^1, v\in \mathcal{G}^0}$ be a relative $\mathcal{G}$-branching system on a measure space $(X,\mu)$. Then, the triple $\{Q_e,B_v, g_e\}$ defined above is a branching systems associated with with $E(\G)$.
\end{proposition}
\begin{proof}
The proof is straightforward and we leave it to the reader.
\end{proof}

\begin{proposition}\label{repcommute} Let $\mathcal{G}$ be a graph,  $V\subseteq \mathcal G^0$, $\{R_e,D_v,f_e\}_{e\in \mathcal{G}^1, v\in \mathcal{G}^0}$ be a relative $\mathcal{G}$-branching system and $\{Q_e,B_v, g_e\}$ as above. Moreover, let $\pi:C^*(\G,V) \to B(L^2(X,\mu))$ be the representation induced from $\{R_e,D_v,f_e\}_{e\in \mathcal{G}^1, v\in \mathcal{G}^0}$, $\rho:C^*(E(\G)) \to B(L^2(X,\mu))$ be the representation induced from $\{Q_e,B_v, g_e\}$, and $\psi: C^*(E(\G)) \rightarrow C^*(\G,V)$ be the isomorphism of Prop.\ \ref{isorelat}. Then, $\rho = \pi \circ \psi$.
\end{proposition}
\begin{proof}
Notice that to show that $\rho = \pi \circ \psi$ is enough to check the equality on the generators of the algebra $C^*(E(\G))$.

Let $v\in \G^0$. It is clear that $\rho(p_v) = \pi \circ \psi (p_v)$. Suppose that $v\notin V$. Let $\phi \in \mathcal{L}^2(X,\mu)$. Then,
\begin{eqnarray*}
  \pi \circ \psi (p_{v'})(\phi)&=& \pi (p_v-\displaystyle \sum_{e:s(e)=v} s_e s_e^*)(\phi)=\chi_{D_v}\cdot \phi-\sum_{e:s(e)=v}\chi_{R_e} \cdot \phi\\
  &=&\chi_{B_{v'}} \cdot \phi=\rho(p_{v'})(\phi).
\end{eqnarray*}

Next, we show that the desired equality holds for the generators associated to edges. Let $e\in \G^1$. If $r(e)\in V$, it is clear that $\rho(s_e) = \pi \circ \psi (s_e)$. Suppose that $v=r(e)\notin V$. Let $\phi \in \mathcal{L}^2(X,\mu)$. Then,

\begin{eqnarray*}
\pi \circ \psi (s_{e})(\phi)&=& \pi (s_e \displaystyle \sum_{f\in s^{-1}(v)}s_fs_f^*)(\phi)=\pi (s_e)\left( \sum_{f\in s^{-1}(v)}\pi(s_fs_f^*)(\phi)\right) \\
&=& \pi (s_e)\left( \chi_{_{\bigcup_{f\in s^{-1}(v)} R_f}} \cdot \phi \right)=  \chi_{_{\bigcup_{f\in s^{-1}(v)} R_f}}\circ f_e^{-1} \cdot \phi\circ f_e^{-1}\\
&=&\chi_{Q_e} \cdot  \phi\circ g_e^{-1}=\rho(s_e)(\phi).
\end{eqnarray*}

The proof that $ \pi \circ \psi (s_{e'})=\rho(s_{e'})$ is analogous. We conclude that $\rho = \pi \circ \psi$ as desired.
\end{proof}

We can now extend \cite[Thm.\ 3.1]{GLR} to representations of relative graph algebras arising from branching systems and obtain an analytical version of \cite[Thm.\  5.6]{CantoGoncalves}. 

\begin{theorem}\label{jacare11}
Let $\mathcal{G}$ be a graph,  $V\subseteq \mathcal G^0$, $\{R_e,D_v,f_e\}_{e\in \mathcal{G}^1, v\in \mathcal{G}^0}$ be a relative $\mathcal{G}$-branching system on a measure space $(X,\mu)$, and let $\pi:C^*(\G,V) \to B(L^2(X,\mu))$ be the representation induced from the relative $\G$-branching system. Then,  $\pi$ is faithful if, and only if, the following conditions are satisfied.

\begin{enumerate}
    \item for each $v\in \G^0$, $D_v$ is non-empty;
    \item $D_v \neq \bigcup_{e\in s^{-1}(v)} R_e$ for all $v\notin V$; and
    \item for each $v \in V$ such that $v$ is a base point of a cycle which has no exit, and for each finite family $\{\alpha^i\}_{i=1}^{n}$ of cycles having $v$ as the base point, there exists a measurable subset $F$ of $D_v$ with $\mu(F) \neq 0$, such that $f_{\alpha^i}(F) \cap F\stackrel{\mu-a.e.}{=}\emptyset$ for all $i$.
\end{enumerate}
\begin{proof}

Let $\rho:C^*(E(\G)) \to B(L^2(X,\mu))$ be the representation induced from the branching system associated with $E(\G)$, as in Def.\ \ref{extendedbs}, and let $\psi: C^*(E(\G)) \rightarrow C^*(\G,V)$ be the isomorphism of Prop.\ \ref{isorelat}. 

By Prop.\ \ref{repcommute}, we have that $\rho = \pi \circ \psi$. Since $\psi$ is an isomorphism, we have that $\pi$ is 1-1 iff $\rho$ is 1-1.

To show that $\rho$ is 1-1 we verify that the branching system $\{Q_e,B_v, g_e\}$ associated with $E(\G)$ satisfies the hypothesis of \cite[Thm.\ 3.1]{GLR}.

Let $v\in \G^0$. If $v\in V$, then $\mu(B_v)\neq 0 $ by Condition (1). Suppose that $v\notin V$. Then, $\mu(B_{v'})\neq 0 $ by Condition (2). Also, $\mu(B_{v})\neq 0 $, since for every edge $e$ we have that $\mu(R_{e})\neq 0 $ (since there is the map $f_e$ from $D_{r(e)}$ to $R_e$ from the definition of a branching system and $\mu(D_v)\neq 0 $ for all $v$).
The final hypothesis in \cite[Thm.\ 3.1]{GLR} that needs to be verified is the same as our Condition (3), but for all vertices of the extended graph. Suppose that $v$ is a vertex such that $v\notin V$. Then, $v'$ is a sink and $v$ is not the base point of a cycle that has no exit. Therefore, the hypothesis in \cite[Thm.\ 3.1]{GLR} is trivially satisfied for such vertices. If $v\in V$ then, as we said, the hypothesis in \cite[Thm.\ 3.1]{GLR} is the same as our Condition (3). We conclude that $\rho$ is 1-1 as desired.

The converse of the theorem follows analogously to what is done in \cite[Thm.\ 5.6]{CantoGoncalves}.

\end{proof}

\end{theorem}

\section{Interval maps as branching systems and the corresponding representations}
\label{sectintmapsBS}

In this section, we show how, from an interval map, to build a branching system such that the induced representations of the associated relative graph algebra coincide.

Suppose that $\g\in M(I)$ is such that $E_{\g}\neq \emptyset$, and let $\mathcal G$ be the graph associated with its transition matrix. Let $x\in E_\g$, and consider the associated relative graph algebra $C^*(\mathcal G,V)$, with $V\subseteq \mathcal G^0$ such that $ \hat{a}_{i\iota(x)} =0$ for $i\in V$. Next, we build the desired relative branching system $(R_{ij}, D_{v_j},f_{ij})$ associated with $\G$.

Define $X= \{z: \g^{k}\left(z\right)= e\left(x\right)\ \hbox{for some}\ k\in \mathbb{N}_{0}\}\setminus \Gamma,$ equipped with the counting measure. 
Notice that, with the current hypothesis on $\g$, we have that $\g(\Gamma)\subseteq \Gamma$ and hence 
\begin{equation}\label{eqXmu}
X= \{z: \g^{k}\left(z\right)= e\left(x\right)\ \hbox{for some}\ k\in \mathbb{N}_{0}\}.
\end{equation}

To simplify notation, from now on denote by $\g_i$ the restriction of $\g$ to $I_i$.

We define a relative branching system on $X$ in the following way.

\begin{definition}\label{branchifs}
For each vertex $j\in \G^0$ and for each edge $ij\in \G^1$ (see \eqref{eqGA}) let \[D_{j}= I_j \cap X \text{ and } R_{ij} = \g_{i}^{-1}( I_j \cap X),\]
and define $f_{ij}:D_{r(ij)}=D_{j}\rightarrow R_{ij}$ as $\g_i^{-1}|_{I_j}$.
\end{definition}

\begin{proposition}\label{propintmapsbsysts} Suppose that $ \hat{a}_{i\iota(x)} =0$ for $i\in V$. Then,
$(R_{ij},D_{j}, f_{ij})$ as above is a $(\G,V)$-relative branching system.
\end{proposition}
\begin{proof}
 
 We have to check conditions (1) to (5) of Def.\ \ref{branchsystem}.
 
Clearly, $I_j\cap I_k \cap X = \emptyset$ for every $i\neq j$ and this implies Conditions~(1) and (2). Conditions~(3) and (5) are straightforward.

To check Condition~(4), let $i$ be a vertex in $V$. Let $\hat{k}=\iota(x)$. Since, by hypothesis, $ \hat{a}_{i\iota(x)} =0$, we have that $Im(\g_i) \subseteq \cup_{l} I_l \cup_{j\neq k} E_j$. Hence,
\[D_i=I_i\cap X = \g_i^{-1}\left(\cup_l I_l \cup_{j\neq k} E_j\right)\cap X = \left(\cup_l  \g_i^{-1}( I_l)\cap X\right)\cup\left(\g_i^{-1}\left(\cup_{j\neq k} E_j\right) \cap X\right) \] 
\[= \bigcup_{l:\g(I_i)\supseteq I_l}  \g_i^{-1}( I_l)\cap X = \bigcup_{l: \g(I_i)\supseteq I_l}  \g_i^{-1}( I_l\cap X)= \bigcup_{il: s(il)=i} R_{il}, \]
where in the equalities above we used that $\g_i^{-1}( E_j) \cap X = \emptyset$ for $j\neq k$ (since if $\g^i(y)\in E_j$, $j\neq k$, then $y\notin X$) and $\g_i^{-1}( I_l)\cap X = \g_i^{-1}( I_l\cap X)$ (since if $\g^i(z)\in X$ then $z\in X$).

Therefore, all conditions of the definition of a relative branching system are satisfied, as desired.

\end{proof}

Notice that the condition $ \hat{a}_{i\iota(x)} =0$ for $i\in V$ is necessary to obtain a representation from an interval map and it was also required above. With this in mind, next we prove the equivalence between the representation associated with an interval map and the one arising from the branching system above.

\begin{theorem}\label{repequivalence} Let $ \nu_x$ be the representation of $C^*(\mathcal{G},V)$ as in Thm.\ \ref{thmnice100}. Let $(R_{ij}, D_{j},f_{ij})$ be the branching system on $X$ given in Def.\ \ref{branchifs} and $\pi:C^*(\mathcal{G},V) \to B(\mathcal{L}^2(X,\mu))$ be the associated representation of $C^*(\mathcal{G},V)$ (as in Prop.\ \ref{repinducedbybranchingsystems}). Then, $ \mathcal{L}^2(X) = H_x $ and $\pi$ and $ \nu_x$  coincide.

\end{theorem}
\begin{proof}

Since the measure on $X$ is the counting measure, it is clear that $H_x=\ell^2(R_\g(x)) =\mathcal{L}^2(X)$. Also, since the measure on $X$ is the counting measure, we have that 
 $\pi(s_{ij})(\phi)= \chi_{R_{ij}} \cdot \phi \circ f_{ij}^{-1}$ and $\pi(p_{j})(\phi)=\chi_{D_{j}}\cdot\phi$, for all $e \in \mathcal{G}^1, v \in \mathcal{G}^0$ and $\phi \in \mathcal{L}^2(X,\mu)$. 
 
 Let $y\in R_\g(x)$ and denote by $\chi_y$ the characteristic function of $y$ in $\mathcal{L}^2(X)$ (which can be identified with the vector $y$ in Dirac notation). Then,
 \begin{eqnarray*}
\pi(s_{ij})(\chi_y)&=& \chi_{R_{ij}} \cdot \chi_y \circ f_{ij}^{-1} = \chi_{\g_i^{-1}(I_j\cap X)}\cdot \chi_{\g_i^{-1}(y)} = \chi_{\g_i^{-1}(I_j\cap X)\cap \g_i^{-1}(y)} \\
&=&\begin{cases} 0 \text{ if } y\notin I_j\cr 1 \text{ if}\ y\in I_j \end{cases} =  \nu_x(s_{ij})(\chi_y),
 \end{eqnarray*}
 and
  \[\pi(p_{i})(\chi_y)= \chi_{D_{i}} \cdot \chi_y  = \chi_{I_i\cap X \cap \{y\}} =\begin{cases} 0 \text{ if $y\notin I_i $} \\ 1 \text{ if $y\in I_1 $} \end{cases}\]
    \[ =  \nu_x(p_i)(\chi_y).\]

So, $\pi$ and $\nu_x$ coincide, as desired.

\end{proof}

\section{Applications}
\label{secapps}

In this section, we apply the previous results to describe faithfulness of representations arising from Markov maps with non-empty escape set.

Let $\g\in M(I)$ such that $E_g\neq \emptyset$ and $\widehat{A}_\g$ be its escape transition matrix (see Def.\ \ref{mmx}). For $x\in E_\g$, let us define the following subset of $ \G_{A_\g}^0$:
\begin{equation}\label{eqvx}
V_x=\{k\in \G_{A_\g}^0: \hat{a}_{k\, \iota(x)}=0\}.
\end{equation}
Thus $\hat{a}_{k\, \iota(x)}=1$ for $k\notin V_x$.

The main result of the section is the following.

\begin{theorem}\label{injectiveMap} Let $\g\in M(I)$, $x\in E_\g$, $V\subseteq \mathcal{G}^0$ such that $ \hat{a}_{i\iota(x)} =0$ for $i\in V$, and $\pi$ be the induced representation of $C^*(\mathcal{G},V)$ associated with the branching system of Def.\ \ref{branchifs}, as in Thm.~\ref{repequivalence}. Then, $\pi$ is faithful if, and only if, the following conditions are satisfied.
\begin{enumerate}[(a)]
    \item  For every $j$, there exists $y\in I_j$ such that $\g^n(y)=e(x)$ for some $n\in \mathbb{N}$.
    \item $V=V_x.$
    \item For each $j \in V$, such that $j$ is a base point of a simple cycle $\alpha=\alpha_1 \ldots \alpha_k$ which has no exit, and for each finite set $\{n_1,\ldots, n_l\}\subseteq \N$, there exists a point $x\in I_j\cap X$ such that $f_{\alpha}^{n_i}(x) \neq x$ for all $i=1, \ldots, l$, that is, such that $$\left( g_{s(\alpha_1)}^{-1}|_{I_{r(\alpha_1)}}\circ \ldots \circ g_{s(\alpha_k)}^{-1}|_{I_{r(\alpha_k)}} \right)^{n_i}(x)\neq x,$$ for all $i=1, \ldots, l$.

\end{enumerate}

\end{theorem}

\begin{proof}

Consider the branching system \[\left(R_{ij} = \g_{i}^{-1}( I_j \cap X), D_{j}= I_j \cap X, f_{ij}= \g_i^{-1}|_{I_J}\right),\] of Def.\ \ref{branchifs}.
The result follows from Thm.~\ref{jacare11}, once we show that the Conditions (1), (2) and (3) of that theorem, with respect to the above branching system, are equivalent to Conditions (a), (b) and (c) above.

Clearly, Condition (a) above is equivalent to Condition (1) in Thm.~\ref{jacare11} and Condition (c) is equivalent to Condition (3) (since we are using the counting measure on $X$). Suppose that Condition (b) holds. To check  Condition~(2), let $\hat{j}=\iota(x)$. Notice that for $k\notin V_x$, since $\hat{a}_{k\iota(x)} =1$, we have that \begin{equation}\label{dk} D_k=\bigcup_i \g_k^{-1}(I_i\cap X) \bigcup \g_k^{-1}(E_{j} \cap X),
\end{equation}
where $\g_k^{-1}(E_{j} \cap X)$ is non-empty.  Since $E_{j} \cap X$ is disjoint from $I_i\cap X$ for every $i$, it follows that $D_k\neq \cup R_{ki}$ as desired. Now, suppose that Condition~(2) holds. We have to prove that $V_x\subseteq V$, since the other inclusion is always true. Suppose that $k\notin V$. Then, from Equation~\ref{dk} and Condition~(2), we have that $\g_k^{-1}(E_{j} \cap X)$ is non-empty, that is, $\hat{a}_{k\, \iota(x)}=1$ (notice that $E_{j} \cap X =  \iota(x)) $. Hence, $k\notin V_x$.
\end{proof}

The above theorem allows us to recover  \cite[Thm.\ 3.1]{RMP20}, with the further improvement of not asking the Markov map to be expansive or aperiodic. 

\begin{corollary}\label{coelho}(cf.\ Thm.\ \ref{rmp20theorem}) Let $\g\in M(I)$ be irreducible and $x\in E_\g$. 
Then, $\nu_x$ is faithful.
\label{cor1}
\end{corollary}
\begin{proof}

We have to check the conditions of 
Thm.\ \ref{injectiveMap}. First we verify (a). Since $x\in E_g$, $e\left( x\right)=\g^{\tau \left( x\right)}\left( x\right) $ and $x\in I_l$ for some $l$. By the irreducibility of $g$, for every $j$ there exists $q\in \N$ such that $g^q(I_j)\supset I_l$. In particular, there is $y\in I_j$ such that $g^q(y)= x$ and hence $g^{\tau(x)+q}(y)=e\left(x\right)$ as desired. 

Condition~(b) of Thm.\ \ref{injectiveMap} follows by definition. If every simple cycle in $\G$ based at a vertex $j\in V$ has an exit, then Condition~(c) is trivially satisfied. Suppose that there exists a vertex $j\in V$ that is the base of a simple cycle and Condition~(c) is not satisfied. Then, there exists a finite set $F=\{n_1,\ldots, n_k\}$ such that for every $z\in I_j\cap X$, there exists $n_j\in F$ such that $z$ is periodic of period $n_j$ with respect to $g$, that is, $g^{n_j}(z)=z$.  Since points in $X$ can not be periodic (since they are in the orbit of the point $x$, which is in the escape set) we conclude that $I_j\cap X = \emptyset$, what contradicts what we proved above (for Condition~(a)).


\end{proof}





\begin{remark}
The concepts of purely atomic, monic, and permutative representations of relative graph algebras have not, as far as we know, been defined in the literature yet. Nevertheless, their definition should be the expected extension of the concepts defined for graph C*-algebras. With this in mind, if $ \nu_x$ is the representation of $C^*(\mathcal{G},V_x)$ as in Thm.\ \ref{thmnice100} associated to an expansive, aperiodic Markov map $\g$, then $\nu_x$ is irreducible (by \cite[Cor.\ 3.7]{RMP20}). Since $\nu_x$ is equal to $\pi:C^*(\mathcal{G},V_x) \to B(\mathcal{L}^2(X,\mu))$, where $\pi$ is  associated to the relative branching system $(R_{ij}, D_{j},f_{ij})$ described in Theorem~\ref{repequivalence}, we obtain that the representation $\rho: C^*(E(\G)) \to B(L^2(X,\mu))$ of the extended graph is irreducible (see Prop.\ \ref{repcommute}). Applying the theory of representations associated with branching systems, we obtain, by \cite[Thm.\ 5.3]{FGG}, that $\rho$ is purely atomic and, by \cite[Thm.\ 5.5]{FGG}, that $\rho$ is monic. Finally, by \cite[Thm.\ 4.10]{FGJKP}, if $\rho$ is supported on an orbit of an aperiodic path, then it is permutative. Of course, the statements about $\rho$ above should be reflected in $\nu_x$, once they have been defined for relative graphs (we do not do it here in the interest of space).

\end{remark}



\section{Examples}

In this section, we present a number of examples to illustrate the improvements obtained with Thm.\ \ref{injectiveMap} and Cor.\ \ref{coelho}.

The following examples exhibits a non-expansive aperiodic Markov interval map $\g\in M(I)$, satisfying Cor.\ \ref{cor1}. 


\begin{example}[A non-expansive and aperiodic Markov interval map] 
Let $\g$ be the piecewise linear map
\[
\g\left( x\right) =\left\{
\begin{array}{l}
\dfrac{5}{11}x+\dfrac{4}{11}\text{ \ \ \ \ \ \ if \ \ }0\leq x\leq c_{1},\\
\\
\dfrac{22}{7}x-\dfrac{9}{14}\text{ \ \ \ \ \ \ if \ \ }c_{1}\leq x\leq
c_{2},\\ \\
\dfrac{55}{3}x-\dfrac{22}{3}\text{\ \ \ \ \ \ \ if \ \ }c_{3}\leq x\leq
c_{4},\\ \\
\dfrac{37}{22}x-\dfrac{21}{44}\text{ \ \ \ \ \ \ if \ \ }c_{5}\leq x\leq
c_{6},\\ \\
\dfrac{5}{11}x+\dfrac{1}{22}\text{ \ \ \ \ \ \ if \ \ }c_{6}\leq x\leq 1,%
\end{array}%
\right.
\]
where $(c_{0},c_{1},c_{2},c_{3},c_{4},c_{5},c_{6},c_{7})$=$(0,3/10,4/11,4/10,5/11,5/10,7/10,1)$.
The graph of $\g$ is as follows:

\centerline{\includegraphics[scale=.380]{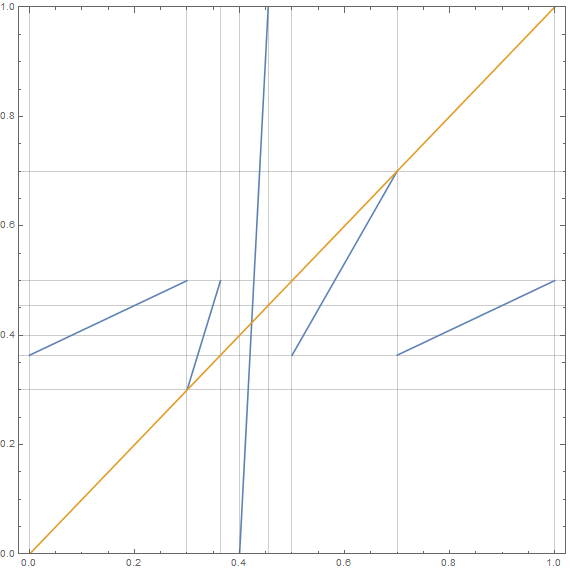}}

We partition the interval $[0,1]$ as follows. Let $I_{1}=\left[ 0,c_{1}\right] $, $I_{2}=\left[ c_{1},c_{2}\right] $, 
$I_{3}=\left[ c_{3},c_{4}\right] $, $I_{4}=\left[ c_{5},c_{6}\right] $, 
$I_{5}=\left[ c_{6},1\right]$, with escape intervals $E_{3}=\left]
c_{2},c_{3}\right[ $, $E_{5}=\left] c_{4},c_{5}\right[ $.

Then, $\g$ is a Markov interval map (not expansive, due to its behavior in the intervals $I_{1}$, $I_{5}$). 

The transition matrix $A_\g$ and the escape transition matrix $\widehat{A}_\g$ are as follows
\[
A_\g=\left(
\begin{array}{ccccccc}
0 & 0  & 1 &  0 & 0 \\
0 & 1  & 1 &  0 & 0 \\
1 & 1  & 1 &  1 & 1 \\
0 & 0  & 1 &  1 & 0 \\
0 & 0  & 1 &  0 & 0%
\end{array}%
\right),\qquad \widehat{A}_\g=\left(
\begin{array}{ccccccc}
0 & 0 & 1 & 1 & 1 & 0 & 0 \\
0 & 1 & 1 & 1 & 1 & 0 & 0 \\
0 & 0 & 0 & 0 & 0 & 0 & 0 \\
1 & 1 & 1 & 1 & 1 & 1 & 1 \\
0 & 0 & 0 & 0 & 0 & 0 & 0 \\
0 & 0 & 1 & 1 & 1 & 1 & 0 \\
0 & 0 & 1 & 1 & 1 & 0 & 0%
\end{array}%
\right).
\]%
The directed graph $\G_{A_\g}$ is as follows
$$
\begin{tikzpicture}[->,>=stealth',shorten >=1pt,auto,node distance=3cm,
                    thick,main node/.style={circle,draw,font=\sffamily\bfseries}]

  \node[main node] (1) {$1$};
  \node[main node] (2) [right of=1] {$2$};
  \node[main node] (3) [below right of=1] {$3$};
  \node[main node] (4) [below right of=3] {$4$};
  \node[main node] (5) [below left of=3] {$5$};

\path [every node/.style={font=\sffamily\small}]
 
(1) edge[->] (3)
(3) edge [bend right] (1)
(2) edge (3)
(3) edge [bend right] (2)
(4) edge (3)
(3) edge [bend right] (4)
(5) edge (3)
(3) edge[bend right] (5)
(2) edge[loop above] (2)
(3) edge[loop right] (3)
(4) edge [loop below]   (4);
\end{tikzpicture}
$$

The elimination of the escape states produces an aperiodic matrix $A_\g$, since
every entry in $A_\g^{2}$ is nonzero.
Since $E_i\cap \g(I_j)\not=\emptyset$ for $i=3,5$ and $j=1,2,3,4,5$, we have $V_x=\emptyset$ whenever $x\in E_\g$. 
We have two (one for each escape interval) representations $\nu_x$ of the \C-algebra $\Cstar(\G_{A_\g},\emptyset)$, which are faithful by Cor.\ \ref{cor1}.
\end{example}


\begin{example}[An non-expansive, irreducible, non-aperiodic,  Markov map with only one cycle]

Let $g$ be the following piecewise linear map with

\begin{equation*}
g\left( x\right) =\left\{
\begin{array}{l}
\dfrac{11}{4}x+\dfrac{1}{5}\text{ \ \ \ \ \ \ \ if \ \ }0\leq x\leq
c_{1}^{-} \\ \\
\dfrac{1}{2}x+\dfrac{5}{8}\text{ \ \ \ \ \ \ \ if \ \ }c_{1}^{+}\leq
x\leq c_{2} \\ \\
\dfrac{4}{5}x-\dfrac{3}{5}\text{\ \ \ \ \ \ \ \ if \ \ }c_{2}\leq x\leq
1%
\end{array}%
\right.
\end{equation*}
where
\begin{equation*}
(c_{0},c_{1}^{-},c_{1}^{+},c_{2},c_{3}) =(0,1/5,1/4,3/4,1).
\end{equation*}

The graph of $g$ is a follows:

\centerline{\includegraphics[scale=.380]{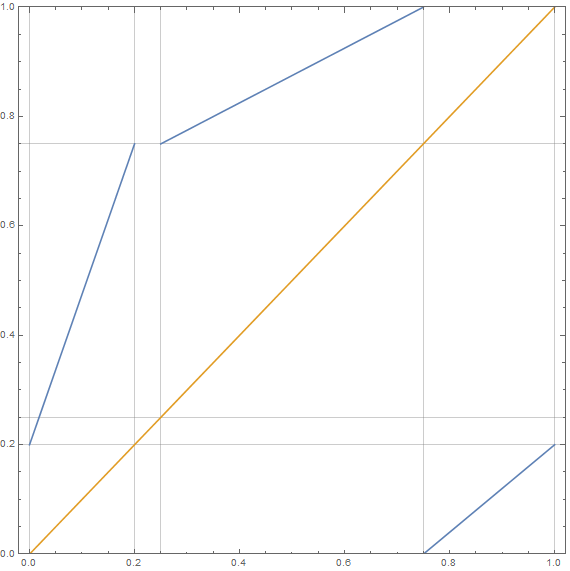}}

Then $g\in M(I)$ with partition being $I_{1}=\left[0,c_{1}^{-}\right] $, $I_{2}=\left[ c_{1}^{+},c_{2}\right]
$, $I_{3}=\left[ c_{2},c_{3}\right] $ and the escape interval set  $E_\g=E_{1}=\left]
c_{1}^{-},c_{1}^{+}\right[$.

Then  the transition matrix $A_g$ and  escape transition matrix $\widehat{A}_g$ are as follows
\begin{equation*}
A_g=\left(
\begin{array}{ccc}
0  & 1 & 0 \\
0  & 0 & 1 \\
1  & 0 & 0%
\end{array}%
\right),\ \ 
\widehat{A}_g=\left(
\begin{array}{cccc}
0 & 1 & 1 & 0 \\
0 & 0 & 0 & 0 \\
0 & 0 & 0 & 1 \\
1 & 0 & 0 & 0%
\end{array}%
\right).
\end{equation*}%

Clearly, $A_g$ is a non-aperiodic and irreducible matrix whose directed graph $\G_{A_g}$ is as follows:
$$\begin{tikzpicture}[->,>=stealth',shorten >=1pt,auto,node distance=3cm,
                    thick,main node/.style={circle,draw,font=\sffamily\bfseries}]

  \node[main node] (1) {$1$};
  \node[main node] (2) [right of=1] {$2$};
  \node[main node] (3) [below of=1] {$3$};

\path [every node/.style={font=\sffamily\small}]

(1) edge[->] (2)
(2) edge[->] (3)
(3) edge[->] (1);
\end{tikzpicture}
$$

In this example, we have $V_x=\{2,3\}$ for $x\in E_\g$ and Cor.\ \ref{cor1} guarantees that $\nu_x$ is a faithful representation of $\Cstar(\G_g,\{2,3\})$. 

\end{example}

\begin{example}[A non-expansive, non-aperiodic, irreducible Markov map with more than one cycle]

Let $g$ be the following map
\begin{equation*}
g\left( x\right) =\left\{
\begin{array}{l}
\dfrac{2}{3}x+\dfrac{1}{5}\text{ \ \ \ \ \ \ \ if \ \ }0\leq x\leq c_{1},
\\   \\
\dfrac{5}{2}x-\dfrac{1}{6}\text{ \ \ \ \ \ \ \ if \ \ }c_{1}\leq x\leq
c_{2}, \\   \\
3x-1\text{\ \ \ \ \ \ \ \ if \ \ }c_{2}\leq x\leq
c_{3}^{-}, \\ \\
\dfrac{15}{4}x-\dfrac{3}{2}\text{ \ \ \ \ \ \ \ if \ \ }c_{3}^{+}\leq
x\leq c_{4}, \\   \\
4x-\dfrac{7}{3}\text{ \ \ \ \ \ \ \ if \ \ }c_{4}\leq x\leq
c_{5}, \\ \\
\dfrac{1}{3}x+\dfrac{5}{12}\text{ \ \ \ \ \ \ \ if \ \ }c_{5}\leq x\leq 1,%
\end{array}%
\right.
\end{equation*}
where $
(c_{0},c_{1},c_{2},c_{3}^{-},c_{3}^{+},c_{4},c_{5},c_{6})
=(0,1/5,3/9,2/5,3/5,2/3,3/4,1).
$

The graph of $g$ is a follows:

\centerline{\includegraphics[scale=.380]{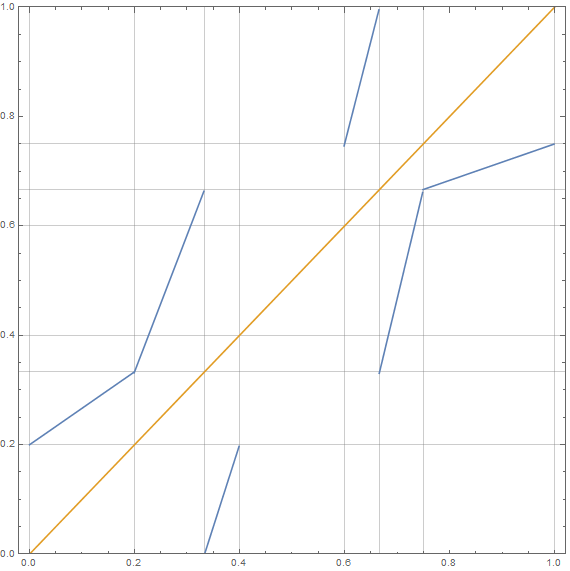}}

Then, $g$ is a Markov map with partition $I_{1}=\left[ 0,c_{1}\right] $, $I_{2}=\left[ c_{1},c_{2}\right] $, 
$I_{3}=\left[ c_{2},c_{3}^{-}\right] $, $I_{4}=\left[ c_{3}^{+},c_{4}\right]$, 
$I_{5}=\left[ c_{4},c_{5}\right] $, $I_{6}=\left[ c_{5},1\right] $, and escape set $E_{3}=\left] c_{3}^{-},c_{3}^{+}\right[$.

The transition matrix $A_g$ and escape transition matrix $\widehat{A}_g$ are as follows:
\begin{equation*}
A_g=\left(
\begin{array}{ccccccc}
0 & 1 & 0  & 0 & 0 & 0 \\
0 & 0 & 1 & 1 & 0 & 0 \\
1 & 0 & 0  & 0 & 0 & 0 \\
0 & 0 & 0  & 0 & 0 & 1 \\
0 & 0 & 1  & 1 & 0 & 0 \\
0 & 0 & 0 & 0 & 1 & 0%
\end{array}%
\right),\ \ 
\widehat{A}_g=\left(
\begin{array}{ccccccc}
0 & 1 & 0 & 0 & 0 & 0 & 0 \\
0 & 0 & 1 & 1 & 1 & 0 & 0 \\
1 & 0 & 0 & 0 & 0 & 0 & 0 \\
0 & 0 & 0 & 0 & 0 & 0 & 0 \\
0 & 0 & 0 & 0 & 0 & 0 & 1 \\
0 & 0 & 1 & 1 & 1 & 0 & 0 \\
0 & 0 & 0 & 0 & 0 & 1 & 0%
\end{array}%
\right).
\end{equation*}%

The associated directed graph $\G_{A_\g}$ is the following:

$$\begin{tikzpicture}[->,>=stealth',shorten >=1pt,auto,node distance=3cm,
                    thick,main node/.style={circle,draw,font=\sffamily\bfseries}]

  \node[main node] (1) {$1$};
  \node[main node] (2) [right of=1] {$2$};
  \node[main node] (3) [below of=1] {$3$};
  \node[main node] (4) [below of=2] {$4$};
  \node[main node] (5) [below of=3] {$5$};
  \node[main node] (6) [below of=4] {$6$};

\path [every node/.style={font=\sffamily\small}]

(1) edge[->] (2)
(2) edge[->] (3)
(3) edge[->] (1)
(2) edge[->] (4)
(5) edge[->] (3)
(5) edge[->] (4)
(4) edge[->] (6)
(6) edge[->] (5);

\end{tikzpicture}
$$
In this case, $V_x=\{1,3,5,7\}$ for $x\in E_g$, and Cor.\ \ref{cor1} guarantees that $\nu_x$ is a faithful representation of $\Cstar(\G_g,\{1,3,5,7\})$.
\end{example}

{\bf Acknowledgments.}
Daniel Gon\c{c}alves was partially supported by Conselho Nacional de Desenvolvimento Cient\'{i}fico e Tecnol\'{o}gico - CNPq and Capes-PrInt, Brazil. Correia Ramos's work was partially supported by national
funds through Funda\c c\~ao Nacional para a Ci\^encia e a Tecnologia (FCT),
Portugal, grant UIDB/ 04674/ 2020 and the Centro de Investiga\c c\~ao em
Matem\'atica e Aplica\c c\~oes, Universidade\ de \'Evora. The work of
Martins and Pinto was partially supported by
FCT/Portugal through CAMGSD, IST-ID, projects UIDB/04459/2020 and
UIDP/04459/2020.

\end{document}